\documentclass[a4paper, 12pt]{amsart}
\usepackage[a4paper,hmargin=2.5cm,vmargin=2.5cm]{geometry}
\usepackage[T2A]{fontenc}
\usepackage[utf8]{inputenc}

\usepackage[russian,english]{babel}
\usepackage{amsmath,amssymb,amsthm,amsfonts,mathrsfs,amscd}
\usepackage[linkcolor=blue,citecolor=blue,colorlinks=true,urlcolor=blue]{hyperref}
\usepackage{graphics}
\numberwithin{equation}{section}
\newtheorem{theorem}{Theorem}[section]
\newtheorem{lemma}[theorem]{Lemma}
\newtheorem{propos}[theorem]{Proposition}
\theoremstyle{definition}
\newtheorem{definition}[theorem]{Definition}

\newtheorem{proof}{Proof}
\let\origendproof\endproof
\def\endproof{\unskip\nobreak\hskip5pt plus 1fill$\square$\origendproof}

\newtheorem{sled}[theorem]{Corollary}
\newtheorem{ex}[theorem]{Example}
\newtheorem{rem}[theorem]{Remark}
\newtheorem{notation}[theorem]{Notation}
\def\Im{\operatorname{Im}}

\def\sep{\mathrm{sep}}
\def\Gal{\mathop{\rm Gal}\nolimits}

\begin{document}
\title{Forms of del Pezzo surfaces of degree 5 and 6}
\author{A.\,V.~Zaitsev}
\address{National research university ''Higher school of economics''}
\email{\href{alvlzaitsev1@gmail.com}{alvlzaitsev1@gmail.com}}

%\date{03.11.2021}
%\udk{512.774.4}
\maketitle

\begin{abstract}
    In this paper we obtain necessary and sufficient condition for existence of del Pezzo surfaces of degree $5$ and $6$ over a field $K$ with a prescribed action of absolute Galois group $\text{Gal} ( K^{\sep}/K)$ on the graph of $(-1)$-curves. Also we compute automorphism groups of del Pezzo surfaces of degree $5$ over arbitrary fields. 
\end{abstract}

\tableofcontents
\section{Introduction}\label{introduction}

Del Pezzo surface is a smooth projective surface $X$ with ample anticanonical class. Automorphism groups of del Pezzo surfaces over algebraically closed fields of characteristic zero were completely described in the paper of I.V. Dolgachev and V.A. Iskovskikh \cite{DI1}, when they studied finite subgroups in birational automorphism group of projective plane. There are only partial results about del Pezzo surfaces automorphisms over an arbitrary field (see \cite{DI2, Y}, also see \cite{Serre}). The ideal result would be the complete description of automorphism groups of del Pezzo surfaces over a prescribed field, analogous to one, obtained in dimension one in \cite{Beauv} and~\cite{G-A}. Although, even the description of groups acting minimally on del Pezzo surfaces will be useful due to possibility of it's application to classification of finite subgroups of the group of birational automorphisms of projective plane over arbitrary fields.

The self-intersection index of the canonical class $K_X$ is called the degree of a del Pezzo surface. This index can take integer values from 1 to 9. In this paper we study del Pezzo surfaces of degrees $5$ and $6$ over various fields.

In the case of an algebraically (or separably) closed field, a del Pezzo surface of degree $d$ is either isomorphic to $\mathbb{P}^1\times\mathbb{P}^1$, or is obtained by blowing up the projective plane at $9 - d$ points in general position. In particular, for $d = 5, 6$, these surfaces are unique up to isomorphism. In the case of an arbitrary field $K$, an additional invariant appears, namely, an action of the Galois group $\text{Gal} ( K^{\sep}/K )$ on the graph of $(-1)$-curves. It is this invariant will be studied in the paper.

%Группа автоморфизмов этой поверхности изоморфна симметрической группе $\mathfrak{S}_5$, см., например,~\cite[Theorem 8.5.8]{D}.

%Над алгебраически (или сепарабельно) замкнутым полем поверхность дель Пеццо степени 5 содержит ровно десять $(-1)$-кривых. Группа автоморфизмов графа пересечений этих кривых изоморфна симметрической группе $\mathfrak{S}_5$, см.~\cite[§8.5.4]{D}.

Consider an arbitrary field $K$. Let $X$ be a del Pezzo surface of degree $5$ over~$K$. The Galois group $\text{Gal} ( K^{\sep}/K )$ acts on the graph of $(-1)$-curves by automorphisms. The automorphism group of this graph is isomorphic to the symmetric group $\mathfrak{S}_5$, see~\hbox{\cite[§8.5.4]{D}}, so that we get the homomorphism $$h\colon\text{Gal}(K^{\sep}/K) \xrightarrow[]{}\mathfrak{S}_5.$$ The same homomorphism corresponds to the action of the group $\text{Gal} (K^{\sep}/K)$ on five conic bundle structures on $X$.

Given a subgroup $H \subset \mathfrak{S}_5$, we denote its conjugacy class by $[H]$. It is clear that the isomorphism class of subgroup does not determine its conjugacy class, for example, the subgroup  $\langle (12) \rangle$ is not conjugate to $\langle (12)(34) \rangle$ (for clarity, all conjugacy classes are written out at the beginning of Section~\ref{preliminaries}).

\begin{definition}\label{def:type5}
We say that a del Pezzo surface of degree $5$ has type $[H]$ if  $\Im{h} \in [H]$.
\end{definition}

Note that the isomorphism class of a del Pezzo surface of degree $5$ is not determined by its type in general, see Remark~\ref{rem:remark1}. The classification of del Pezzo surfaces of degree $5$ up to isomorphism over fields of characteristic zero is described in~\cite[Theorem~3.1.3]{Skoro}; see also~\cite[Proposition 4.7(iv)]{SLZ}.

In this paper we will prove the following theorem.

\begin{theorem}\label{theo:theo}
There exists a del Pezzo surface of degree $5$ of type $[H]$ over $K$ if and only if there exists a Galois extension of fields $L\supset K$ with Galois group isomorphic to $H$.
\end{theorem}

For some important types of fields, Theorem~\ref{theo:theo} gives a simple necessary and sufficient condition for existence of a del Pezzo surface of degree $5$ of a given type.

\begin{sled}\label{sled:Q()}
Let $\mathbb{F}$ be a number field. Then there exist del Pezzo surfaces of degree $5$ of all types over $\mathbb{F}$.
\end{sled}

\begin{sled}\label{sled:F_q}
Let $\mathbb{F}$ be a finite field. Then there exists a del Pezzo surface of degree $5$ of type $[H]$ over $\mathbb{F}$ if and only if $H$ is a cyclic group.
\end{sled}

Note that for fields of characteristic zero Theorem~\ref{theo:theo} can be derived from the more accurate result~\mbox{\cite[Theorem~3.1.3]{Skoro}}. We prove it in a different way using the more elementary and more geometric approach that works for arbitrary fields, in particular, imperfect ones. In addition, the advantage of a rougher classification, which is given by Theorem~\ref{theo:theo}, is that it is convenient to describe automorphism groups of del Pezzo surfaces of degree $5$ in such terms.

The second main result of this paper is the classification of automorphism groups of del Pezzo surfaces of degree $5$ depending on their type. Recall that over an algebraically closed field, the automorphism group of the del Pezzo surface of degree $5$ is isomorphic to the group $\mathfrak{S}_5$.

\begin{theorem}\label{theo:autgrp}
The automorphism group of a del Pezzo surface of degree $5$ of type $[H]$ is isomorphic to the centralizer of the subgroup $H$ in the group $\mathfrak{S}_5$.
\end{theorem}

An explicit list of automorphism groups of del Pezzo surfaces of degree $5$ can be found in~\ref{aut groups}. Also note that Theorem~\ref{theo:autgrp} allows us to obtain for each subgroup $G\subset\mathfrak{S}_5$ a criterion for the existence of a $G$-minimal del Pezzo surface of degree $5$ over a given field (recall that a del Pezzo surface $X$ with the action of the group $G$ is called \emph{$G$-minimal} if $\text{rkPic}(X)^G=1$). Namely, the following proposition will be proved in the paper.

\begin{propos}\label{predl:G-min}
Let $G$ be a subgroup of $\mathfrak{S}_5$. There exists a $G$-minimal del Pezzo surface of degree $5$ if and only if one of the two conditions is satisfied.

\begin{enumerate}
    \item The group $G$ contains a subgroup isomorphic to $\mathbb{Z}/5\mathbb{Z}$ $($note that such group~$G$ can act faithfully only on surfaces of types $[e]$ and $[\mathbb{Z}/5\mathbb{Z}])$.
    \item The group $G$ is trivial and there exists a Galois extension of fields $L\supset K$ such that the group $\Gal(L/K)$ is isomorphic to a subgroup in $\mathfrak{S}_5$ and contains a subgroup isomorphic to $\mathbb{Z}/5\mathbb{Z}$.
\end{enumerate}
\end{propos}

By analogy with del Pezzo surfaces of degree $5$, one can define the type of del Pezzo surface of degree $6$ over a field $K$: this is the conjugacy class of the image of the group $\Gal(K^{\sep}/K)$ in the group $\mathfrak{S}_3\times\mathbb{Z}/2\mathbb{Z}$, for more details, see Section~\ref{dP6}. As a consequence of Theorem~\ref{theo:theo}, we will prove the following theorem.

\begin{theorem}\label{theo:theo2}
There exists a del Pezzo surface of degree $6$ of type $[H]$ over $K$ if and only if there exists a Galois extension of fields $L\supset K$ with Galois group isomorphic to~$H$.
\end{theorem}

The plan of the paper is as follows. In Section~\ref{preliminaries} we will collect some general statements. In Sections~\ref{first} and~\ref{mainc}, we consider ways to construct del Pezzo surfaces of degree $5$. In Section~\ref{affine line}, we formulate and prove the key lemmas for proving Theorem~\ref{theo:theo}. In Sections~\ref{inf fields} and~\ref{finite field}, we prove Theorem~\ref{theo:theo} for infinite and for finite fields, respectively. Then, in Section~\ref{aut groups}, we will prove Theorem~\ref{theo:autgrp} and Proposition~\ref{predl:G-min}. And finally, in Section~\ref{dP6}, we prove Theorem~\ref{theo:theo2}.

We will use the following notation. Let $X$ be an algebraic variety over a field $K$ and $L\supset K$ be a field extension. Then by $X_L$ we denote the extension of scalars of $X$ to $L$. By $\overline{K}$ we denote the algebraic closure of the field $K$, and by $K^{\sep}$ we denote its separable closure.

I would like to thank my advisor Constantin Shramov for stating the problem, useful discussions
and constant attention to this work. I also want to thank Andrey Trepalin for his advice about surfaces of degree 6.

\section{Preliminaries}\label{preliminaries}

In this section, we will collect some (well known) general statements that will be convenient to refer later. But first we list all conjugacy classes of subgroups in the group $\mathfrak{S}_5$ and fix the notation:

\begin{itemize}
    \item the class $[e]$ of the trivial subgroup;
    \item the class $[\langle(1,2)\rangle]$ of a subgroup of order 2 generated by a transposition;
    \item the class $[\langle(1,2)(3,4)\rangle]$ of a subgroup of order 2 generated by a product of two disjoint transpositions;
    \item the class $[\langle(1,2),(3,4)\rangle]$ of a subgroup of order 4 generated by a pair of disjoint transpositions;
    \item the class $[\langle(1,2)(3,4),(1,3)(2,4)\rangle]$ of a subgroup of order 4 generated by two different products of two disjoint transpositions;
    \item classes $[\mathbb{Z}/3\mathbb{Z}]$, $[\mathbb{Z}/4\mathbb{Z}]$, $[\mathbb{Z}/5\mathbb{Z}]$ and $[\mathbb{Z}/6\mathbb{Z}]$ of cyclic subgroups of order 3, 4, 5 and~6;
    \item classes $[\mathrm{D_4}]$, $[\mathrm{D_5}]$ of dihedral subgroups of 8 and 10 elements;
    \item two classes $[\langle(1,2,3),(1,2)\rangle]$ and $[\langle(1,2,3),(1,2)(4,5)\rangle]$ of subgroups isomorphic to~$\mathfrak{S}_3$;
    \item the class $[\mathfrak {S} _3\times\mathbb{Z} / 2\mathbb{Z}] $ of a subgroup isomorphic to the direct product \hbox{$\mathfrak{S} _3\times\mathbb {Z} / 2\mathbb{Z} $};
    \item classes $[\mathfrak{A}_4]$, $[\mathfrak{A}_5]$ of alternating subgroups of degrees 4 and 5;
    \item classes $[\mathfrak{S}_4]$, $[\mathfrak{S}_5]$ of symmetric subgroups of degrees 4 and 5;
    \item the class $[\mathrm{GA}(1,5)]$ of the group $\mathrm{GA}(1,5)\simeq  \mathbb{Z}/5\mathbb{Z}\rtimes\mathbb{Z}/4\mathbb{Z}$.
 \end{itemize}  
There are 19 classes in total.

\begin{rem}\label{rem:maximal subgroups not cont Z_5}
Subgroups $\mathfrak{S}_3 \times \mathbb{Z}/2\mathbb{Z}$ and $\mathfrak{S}_4$ are maximal (with respect to set inclusion) among those that do not contain an element of order five.
\end{rem}

\begin{lemma}\label{lemma:K^sep}
Let $X$ be a del Pezzo surface over a field $K$. Then any $(-1)$-curve on $X_{\overline{K}}$ is defined over $K^{\sep}$.
\end{lemma}

\begin{proof}
Let $\mathcal{H}$ be the Hilbert scheme of $(-1)$-curves on $X$. Let's prove that this scheme is smooth. To do this, we calculate the tangent space at the point $[l] \in \mathcal{H}$ corresponding to $(-1)$-curve $l\simeq\mathbb{P}^1$. The tangent space $T_{[l]}\mathcal{H}$ is isomorphic to $H^0(l,\mathcal{N}_{l/X})$, see~\hbox{\cite[Chapter VI, §4, Theorem 4]{Shaff}}. Thus, we obtain $$T_{[l]}\mathcal{H} \simeq H^0(\mathbb{P}^1, \mathcal{O}_{\mathbb{P}^1}(-1))=0.$$ Hence, the scheme $\mathcal{H}$ is smooth. Therefore, $K^{\sep}$-points are dense in $\mathcal{H}_{\overline{K}}$, see~\hbox{\cite[\href{https://stacks.math.columbia.edu/tag/056 U }{Lemma 056U}]{stacks-project}}. But $\mathcal{H}_{\overline{K}}$ is finite, hence all points of the scheme $\mathcal{H}_{\overline{K}}$ are defined over~$K^{\sep}$.\end{proof}

\begin{lemma}\label{lemma:nec2} 
Let $K$ be a field. Let $X$ be a del Pezzo surface of degree $5$ of type~$[H]$ over $K$.
Then there is a Galois extension of fields $L\supset K$ with Galois group isomorphic to $H$, and each $(-1)$-curve on $X_{\overline{K}}$ is defined over $L$.
\end{lemma}

\begin{proof}
Let's consider the separable closure of our field. From Lemma~\ref{lemma:K^sep} we know that every $(-1)$-curve on $X_{\overline{K}}$ is defined over $K^{\sep}$. Hence, the group $\text{Gal} (K^{\sep}/K)$ acts on the graph of $(-1)$-curves by automorphisms, from this action we obtain the homomorphism $h\colon\text{Gal}(K^{\sep}/K)\xrightarrow[]{} \mathfrak{S}_5$. Denote by $G$ the kernel of this homomorphism. Put $L=(K^{\sep})^G$, then $L\supset K$ is a finite Galois extension with the Galois group $\Gal(L/K)\simeq H$, see~
\hbox{\cite[\href{https://stacks.math.columbia.edu/tag/0BML}{Theorem 0BML}]{stacks-project}}. Since each $(-1)$-curve is invariant with respect to the action of the group $G$, then each $(-1)$-curve on $X_{\overline{K}}$ is defined over~$L$.\end{proof}

\begin{definition}\label{def:generalposition} 
Let $K$ be a field. A set of different $K$-points of the projective plane $\mathbb{P}^2_K$ is called points in general position if no three of them are collinear.
\end{definition}

\begin{lemma}\label{lemma:freeaction}
Let $X$ be a del Pezzo surface of degree $d\leq 5$ over a field $K$. Then the natural action of the group \emph{Aut}$(X)$ on the graph of $(-1)$-curves of the surface $X_{K^{\sep}}$ by automorphisms is faithful.
\end{lemma}

\begin{proof}
Let's show that the action is faithful. Denote the kernel of this action by $G$. Then there exists $G$-equivariant morphism $$X_{K^{\sep}} \xrightarrow[]{}\mathbb{P}_{K^{\sep}}^2,$$ contracting $9 - d \geq 4$ disjoint $(-1)$-curves. The image of these $9 - d$ curves will be $9 - d$ points in general position. But we know that any automorphism of the projective plane fixing four points in general position is trivial. So $G$ is trivial and the action is faithful.

%Выберем четыре непересекающиеся $(-1)$-кривые: $E_1$, $E_2$, $E_3$, $E_4$. Обозначим через $\sigma$ стягивание этой четверки $(-1)$-кривых на плоскость $\mathbb{P}_{K^{\sep}}^2$. Образами $(-1)$-кривых $E_1$, $E_2$, $E_3$, $E_4$ при этом стягивании будут четыре точки общего положения $P_1$, $P_2$, $P_3$, $P_4$ соответственно. Рассмотрим отображение  $$\tau=\sigma \circ f \circ \sigma^{-1} \colon \mathbb{P}^2_{K^{\sep}} \rightarrow \mathbb{P}^2_{K^{\sep}}.$$ Доопределяя это отображение в точках неопределенности следующим образом $$\tau(P_1) = P_1,\, \tau(P_2) = P_2,\,
%\tau(P_3) = P_3,\, \tau(P_4) = P_4,$$ мы получим рациональное отображение, определенное во всех точках, то есть регулярное отображение. Это отображение обратимо, поэтому $\tau$ "--- автоморфизм проективной плоскости. Поскольку $\tau$ оставляет четыре точки общего положения на месте, мы заключаем, что $\tau$ "--- тождественное отображение. Отсюда мы делаем вывод, что автоморфизм $f$ совпадает с тождественным отображением на открытом множестве~\mbox{ $\sigma^{-1}(\mathbb{P}^2_{K^{\sep}}\setminus\{P_1,P_2,P_3,P_4\})$}. Значит, $f$ "--- тождественный автоморфизм, поэтому действие эффективное.
\end{proof}

\begin{figure}[b!]
\centerline{\includegraphics{a-5.mps}}
\caption{\unskip}\label{ris2}	
\end{figure}

Denote by $\Gamma$ the intersection graph of $(-1)$-curves on a del Pezzo surface of degree $5$ over an algebraically closed field. The graph $\Gamma$ is a Kneser graph $KG_{5,2}$, see Fig.~\ref{ris2}. That is, its vertices correspond to two-element subsets of a five-element set, and two vertices are adjacent if and only if the two corresponding sets are disjoint. Its automorphism group is isomorphic to the symmetric group $\mathfrak{S}_5$.

\begin{sled}[see {\cite[Theorem 8.5.8]{D} or \cite[Proposition 3.4]{S2}}]\label{sled:S_5}
The automorphism group of a del Pezzo surface of degree $5$ embeds in the symmetric group~$\mathfrak{S}_5$.
\end{sled}

We will need the following lemma in Section~\ref{dP6} when proving Theorem~\ref{theo:theo2}.

\begin{lemma}\label{lemma:dP5->dP6}
Let $\Gamma$ be the graph of $(-1)$-curves on a del Pezzo surface of degree~$5$. Let $G\simeq\mathfrak{S}_3\times\mathbb{Z}/2\mathbb{Z}$ be a subgroup in $\mathrm{Aut}(\Gamma)\simeq\mathfrak{S}_5$. Then there exists a vertex $v$ of the graph that is invariant with respect to the action of the group $G$.
\end{lemma}

\begin{proof}
Note that the graph $\Gamma$ has the vertex that is invariant with respect to the action of the group $G' = \langle(1,2,3), (1,2), (4,5) \rangle$. Indeed, this is the vertex corresponding to the subset $\{4,5\}$, see Fig.~\ref{ris2}. The groups $G$ and $G'$ are conjugate in $\mathfrak{S}_5$, i.e. there is a permutation $\sigma\in \mathfrak{S}_5$ such that $G' = \sigma^{-1}G\sigma$. Hence the vertex corresponding to the subset $\{\sigma(4),\sigma(5)\}$ is invariant with respect to the action of the group~$G$.\end{proof}

The following lemma is proved by similar reasoning.

\begin{lemma}\label{lemma:no G-inv vert}
Let $\Gamma$ be the graph of $(-1)$-curves on a del Pezzo surface of degree~$5$. Let $G$ be a subgroup in $\mathrm{Aut}(\Gamma)\simeq\mathfrak{S}_5$. Suppose there is no $G$-invariant set of vertices of the graph $\Gamma$, such that no two of this vertices are connected by an edge. Then $G$ contains a subgroup isomorphic to $\mathbb{Z}/5\mathbb{Z}$.
\end{lemma}

\begin{proof}
Suppose that $G$ does not contain an element of order $5$, we can assume that $G$ is maximal with this property. Then, by Remark~\ref{rem:maximal subgroups not cont Z_5}, $G\simeq\mathfrak{S}_3\times\mathbb{Z}/2\mathbb{Z}$ or $G\simeq\mathfrak{S}_4$. In the first case, there is a $G$-invariant vertex by Lemma~\ref{lemma:dP5->dP6}, in the second case there is a $G$-invariant set of four vertices, no two of which are connected by an edge. We got a contradiction, so the group $G$ contains a subgroup isomorphic to $\mathbb{Z}/5\mathbb{Z}$. \end{proof}

\section{Construction of del Pezzo surfaces of degree 5. The first construction}\label{first}

In this section we will construct del Pezzo surfaces of degree $5$ of some types by blowing up the projective plane at a closed set. From this closed set it is required that, after the transition to a separable closure, it becomes a union of four points in general position.

\begin{ex}\label{ex:1}
\emph{The surface of type $[e]$ can be constructed over any field $K$. To do this, we choose four points in general position on the projective plane $\mathbb{P}^2_K$ (such exist even over a field of two elements). Blowing up the plane at the chosen four points, we get a del Pezzo surface of degree $5$. Moreover, all $(-1)$-curves of the constructed surface are defined over $K$, so the group $\Gal(K^{\sep}/K )$ acts on the graph of $(-1)$-curves trivially. Therefore, the constructed surface has the type $[e]$.}
\end{ex}

\begin{ex}\label{ex:2}
\emph{A surface of the type $[\langle(1,2)\rangle]$ can no longer be constructed over any field. Let $K$ be a field, and let $L\supset K$ be a Galois extension of fields with Galois group isomorphic to $\mathbb{Z}/2\mathbb{Z}$. Denote by $\omega$ the element generating the extension~\mbox{$L\supset K$}. By $\Bar{\omega}$, we denote the image of $\omega$ under action by nontrivial element of the Galois group $\Gal(L/K)$. Consider the four points in the projective plane~\mbox{$\mathbb{P}^2_{K^{\sep}}$}: $$(1:0:0),\; (0:1:0),\; (1:\omega:1),\; (1:\Bar{\omega}:1).$$ Not all of these points are defined over~$K$, nevertheless this four forms a closed set defined over~$K$, and we can blow up the projective plane $\mathbb{P}^2_K$ at this closed set. The image of the homomorphism $h$ introduced in Section~\ref{introduction} will be a subgroup of order 2, and more precisely, a subgroup generated by a transposition.
Thus, the constructed del Pezzo surface of degree $5$ has the type $[\langle(1,2)\rangle]$}.
\end{ex}

\begin{rem}\label{rem:remark1}\rm 
There exist non-isomorphic del Pezzo surfaces of degree $5$ of the same type. Indeed, consider a field $K$ that has two distinct quadratic Galois extensions $L\supset K$ and $L'\supset K$. For each extension, you can construct a surface of the type $[\langle(1,2)\rangle]$ in the way described in the Example~\ref{ex:2}. The resulting del Pezzo surfaces of degree $5$ have the same type, but they are not isomorphic.
\end{rem}

\begin{ex}\label{ex:3}
\emph{Let's use the notation of the Example~\ref{ex:2} and consider the four points in the projective plane $\mathbb{P}^2_{K^{\sep}}$: $$(1:\omega:0),\; (1:\Bar{\omega}:0),\; (1:0:\omega),\; (1:0:\Bar{\omega}).$$ This four forms a closed set defined over $K$. Blowing up the projective plane $\mathbb{P}^2_{K}$ in this closed set, we get a del Pezzo surface of degree $5$ of type $[\langle(1,2)(3,4)\rangle]$.}
\end{ex}

\begin{ex}\label{ex:5}
\emph{Let $K$ be a field and suppose we have a triple of points in general position on the projective plane $\mathbb{P}^2_{K^{\sep}}$, invariant with respect to the action of the group $\Gal(K^{\sep}/K)$. Suppose we also know that the image of action of the group $\Gal(K^{\sep}/K)$ in the permutation group of these three points is isomorphic to $\mathbb{Z}/3\mathbb{Z}$. Let's add to this triple a $K$-point such that the resulting four points are in a common position over $K^{\sep}$ (any $K$-point will be suitable, since the original three points are in general position over $K^{\sep}$ and permute in a cycle, that is, none of the three lines passing through two of this three points can contain $K$-point). The resulting four points form a closed set defined over $K$. Blowing up the projective plane $\mathbb{P}^2_{K}$ at this closed set, we get a del Pezzo surface of degree $5$ of type $[\mathbb{Z}/3\mathbb{Z}]$.}
\end{ex}

\begin{rem}\label{rem:first costruction}
We can obtain only del Pezzo surfaces of degree $5$ of type $[H]$, where $H$ is contained in $G\subset \mathfrak{S}_5$, $G\simeq S_4$ via this construction.
\end{rem}

\section{Construction of del Pezzo surfaces of degree 5. The second construction}\label{mainc}

Consider an arbitrary field $K$. Let $C$ be a smooth rational conic in the projective plane $\mathbb{P}^2_{K}$. Suppose that there is a five $K^{\sep}$-points $P_1$,~$P_2$, $P_3$, $P_4$, $P_5$  on $C$, invariant with respect to the action of the group $\text{Gal} ( K^{\sep}/K)$. Since these five points are invariant with respect to the action of the group $\text{Gal} (K^{\sep}/K)$, then the union of these five points is a closed set $Z$ defined over $K$.

Let's blow up the plane $\mathbb{P}^2_{K}$ in the set $Z$ and blow down the strict transform of the conic $C$. We get a surface $X$, which is a del Pezzo surface of degree $5$. Later, in Section~\ref{finite field}, we will see that this construction gives all possible types of del Pezzo surfaces of degree $5$, except the surface of type $[e]$ over fields of two and three elements.

%Впоследствии в разделе~\ref{finite field} мы увидим, что эта конструкция дает все возможные типы поверхностей дель Пеццо степени 5 над $K$, если поле $K$ содержит не меньше пяти элементов. А, например, над полем из двух элементов поверхность дель Пеццо степени 5 типа $[e]$, с помощью такой конструкции построить не получится, так как на конике не найдется пяти $K$-точек.

From this construction we clearly see an isomorphism between group $\mathfrak{S}_5$ and automorphism group of the graph of $(-1)$-curves on $X_{K^{\sep}}$. When we permute the five specified points on the conic, we also permute ten lines, connecting these points, and, as a consequence, we permute $(-1)$-curves preserving intersections on the blow-up surface, that is, we get an automorphism of the graph. The only permutation of points that acts trivially is the identity permutation. Thus, the specified action sets the isomorphism of the group $\mathfrak{S}_5$ with the automorphism group of the graph of $(-1)$-curves.

\begin{rem}
\rm Let $X$ be a del Pezzo surface of degree $5$ over a field $K$. It is well-known, that $X$ has a $K$-point, see, for example,~\cite{S2}. Suppose there is a $K$-point that does not lie on $(-1)$-curves (such a $K$-point always exists if the field $K$ is infinite, since in this case $K$-points on $X$ are Zariski dense (cf. the proof \hbox{\cite[Theorem 4.4]{S2})}, but, for example, there is no such point on the del Pezzo surface of degree $5$ of type $[e]$ over a field of two elements). Then the construction described above is reversible. Namely, you can blow up a point that does not lie on $(-1)$-curves and get surface $X'$, which is a del Pezzo surface of degree 4. On $X'_{K^{\sep}}$ there will be five non-intersecting $(-1)$-curves invariant with respect to the action of the group $\Gal(K^{\sep}/K)$ (it consists of exactly $(-1)$-curves intersecting the exceptional curve $l$ of the blow up $X'\xrightarrow{}X$). Hence, the contraction of $X'$ to $\mathbb{P}^2_K$ is defined, and the image of~$l$ is a smooth rational conic.

This reason proves that over an infinite field any del Pezzo surface of degree $5$ can be obtained by the described construction. This consideration can also be used for an alternative proof of the necessary condition in Theorem~\ref{theo:theo} over infinite fields \hbox{(cf.~\cite[§10]{AB})}.

%Более того, $K$-точки на $X$ плотны по Зарисскому для бесконечных $K$ (ср. с доказательством \hbox{\cite[Theorem 4.4]{S2})}. Поэтому, если поле $K$ бесконечно, то на $X$ есть точка, не лежащая на $(-1)$-кривых (отметим, что на поверхности дель Пеццо степени $5$ типа $[e]$ над полем из двух элементов такой точки нет). Отсюда видно, что для бесконечного поля $K$ описанная выше конструкция обратима. А именно, можно раздуть точку, не лежащую на  $(-1)$-кривых, получить $X'$ --- поверхность дель Пеццо степени 4. На $X'_{K^{\sep}}$ найдется пятерка не пересекающихся $-1$-кривых, инвариантная относительно действия группы $\Gal(K^{\sep}/K)$ (она состоит ровно из $-1$-кривых, пересекающих исключительную кривую $l$ раздутия $X' \xrightarrow{}X$). Следовательно определено стягивание $X'$ на $\mathbb{P}^2_K$, причем образом $l$ является рациональная коника. 

%Это соображение можно использовать для альтернативного доказательства необходимого условия в теореме~\ref{theo:theo} над бесконечными полями \hbox{(ср.~\cite[§10]{AB})}.   
\end{rem}

\section{Points of the affine line and the action of a Galois group}\label{affine line}

As we have just seen, to construct a del Pezzo surface of degree $5$ of type $[H]$, it is enough to present five points on smooth rational conic whose permutations under the action of the Galois group $\text{Gal} ( K^{\sep}/K)$ form a subgroup of $R\subset{\mathfrak{S}_5}$, $R\in[H]$. Since smooth rational conic is isomorphic to a projective line, it is enough for us to present this five points on $\mathbb{P}^1$ and in fact even on the affine line $\mathbb{A}^1$.

The idea of the proof of the following lemma was taken from Jeremy Rouse's answer in the discussion~\cite{act} on the site \texttt{https://mathoverflow.net}.

\begin{lemma}\label{lemma:roots}
Suppose we have a natural number $n$, a group $G$, and a transitive action of the group $G$ on the set $\{1,2,\dots, n\}$. Suppose we also have a Galois extension of fields $L\supset K$ with Galois group isomorphic to $G$. Denote by $H$ the stabilizer of the element $1$ in $G$ and consider the field $M = L^H$. Then the following assertions hold.
\begin{enumerate}

\item The extension $M\supset K$ has degree $n$.

\item There exists an element $\beta_1 \in M$ such that $M = K(\beta_1)$.

\item Let $\mu(x)$ be the minimal polynomial of the element $\beta_1$ over $K$. Then the action of the Galois group $\Gal(K^{\sep}/K)$ on the roots of the polynomial $\mu(x)$ is equivalent to the action of the group $G$ on the set $\{1,2,\dots, n\}$. 

\item One can number the roots $\beta_1$, $\beta_2$, $\dots$, $\beta_n$ of the polynomial $\mu(x)$ in such a way that for any $g\in G$ and for any $i\in \{1,2,\dots, n\}$ the equality $g(\beta_i) = \beta_{g(i)}$ holds.
\end{enumerate}
\end{lemma}

\begin{proof}
The degree of the extension $M\supset K$ is calculated as follows $$[M:K] = \frac{[L:K]}{[L:M]} = \frac{|\Gal(L/K)|}{|\Gal (L/M)|} = \frac{|G|}{|H|} = n.$$
Since the extension $L\supset K$ is separable, the extension $M\supset K$ is separable too. Then by the primitive element theorem \cite[§40]{prim}, the extension $M\supset K$ is simple, that is~\mbox{$M=K(\beta_1)$}. Thus, assertions 1 and 2 are proved.

Since the extension $L\supset K$ is normal, the roots of the polynomial $\mu(x)$ lie in the field~\mbox{$L$}. So the Galois group $\text{Gal} (K^{\sep}/L)$ acts trivially on the roots, and all nontrivial action comes from the group $\text{Gal}(L/K)\simeq G$.

The action of the group $G$ on the set $\{1,2,\dots, n\}$ is equivalent to the action on the left cosets of the subgroup $H$. Let's show that the action of the group $G$ on the roots of $\mu(x)$ is equivalent to the action on the left cosets of the same group $H$.

Denote by $S$ the stabilizer of the root $\beta_1$. Since the action of the group $G$ on the roots of the polynomial $\mu(x)$ is transitive (indeed, we can map the root~$\beta_1$ to any of the roots of the polynomial $\mu(x)$), then $[G:S] = n$.

Since $\beta_1$ lies in the field $L^H$, then the subgroup $H$ is contained in the subgroup $S$, moreover, the index of the subgroup $H$ in the group $G$ also equals $n$, therefore, $H = S$. The action of the group $G$ on the roots of the polynomial $\mu(x)$ is equivalent to the action of $G$ on the left cosets of the subgroup $H$, and this proves assertion 3.

Let's number the left cosets $H_1 = H$, $H_2$, $\dots$, $H_n$ in such a way that for any $i\in \{1,2,\dots,n\}$ the class $H_i$ maps element $1$ to element $i$, in this numbering for any $g \in G$ and for any $i\in \{1,2,\dots, n\}$ the equality $gH_i = H_{g(i)}$ holds. Now we number the roots as follows
$$\vcenter{\openup\jot\halign{\hfil$#$\hfil\cr
\beta_2 = h_2(\beta_1){,}\ h_2 \in H_2;\cr
\beta_3 = h_3(\beta_1){,}\ h_3 \in H_3;\cr
%\multispan1{\dotfill}\cr
\dots\cr
\beta_n = h_n(\beta_1){,}\ h_n \in H_n.\cr}}$$
It is easy to see that in the introduced notation for any $g\in G$ and for any $i\in \{1,2,\dots, n\}$, the equality $g(\beta_i) = \beta_{g(i)}$ holds, thereby, assertion 4 is proved.\end{proof}

\begin{definition}\label{op:def2}
Let $G$ be a subgroup of the symmetric group $\mathfrak{S}_n$. The maximum number of orbits of the same length under the action of $G$ on the set $\{1,2,\dots, n\}$ we call the complexity $c(G)$ of the group $G$.
\end{definition}

\begin{lemma}\label{lemma:g(b_i)} 
Let $n$ and $k$ be natural numbers, let $G$ be a group embedded in~$\mathfrak{S}_n$ of complexity $k$. Suppose a field $K$ has at least $k+1$ elements. If there exists a Galois extension $L\supset K$ with Galois group isomorphic to $G$, then there is a set of $n$ points  $B = \{\beta_1, \beta_2, \dots, \beta_n\}$ on the affine line over $L$ such that for any $g\in G$ and for any $i\in \{1,2,\dots, n\}$, the equality $g(\beta_i) = \beta_{g(i)}$ holds.
\end{lemma}

\begin{proof}

We will prove by induction on the number of orbits of the actions of $G$ on the set $\{1,2,\dots,n\}$.

Base: one orbit.

Since there is only one orbit, then $G$ is a transitive subgroup, and we already know how to present the desired set of points. Indeed, we take the set $B$ сonsists of roots of the polynomial $\mu(x)$ defined in Lemma~\ref{lemma:roots}.

Step: suppose we are able to present the desired set for an action with $m$ orbits, let's present a set for an action with $m + 1$ orbits.

Suppose the set $\{1,2,\dots,n\}$ decompose into $m+1$ orbits under the action of the group $G$. Let's choose one of the orbits of the minimal length $\{i_1, i_2, \dots, i_l\}$, here $l$ is the length of the chosen orbit. The group $G$ acts transitively on the set $\{i_1, i_2, \dots, i_l\}$, therefore, according to the induction base, there is a set of $l$ points $\{\beta'_{i_1}, \beta'_{i_2}, \dots, \beta'_{i_l}\}$ on the affine line over $L$ such that for any $g\in G$ and for any $s\in \{1,2,\dots, l\}$ the equality $g(\beta'_{i_s}) = \beta'_{g(i_s)}$ holds. Moreover, we can assume that $\beta'_{i_1}$ is not equal to zero (this is always the case if $l > 1$, if $l = 1$, then you can take $1$ as $\beta'_{i_1}$).

Now consider the set $\{1,2,\dots,n\} \setminus\{i_1, i_2, \dots, i_l\}$, denote it $\{j_1, j_2,\dots, j_{n - l}\}$. On the set $\{j_1, j_2, \dots, j_{n - l}\}$, the group $G$ acts with $m$ orbits, so, by induction assumption, there will be a set of $n - l$ points $\{\beta_{j_1}, \beta_{j_2}, \dots, \beta_{j_{n-l}}\}$ on the affine line over $L$ such that for any $g\in G$ and for any $s\in \{1,2,\dots, n-l\}$ the equality $g  (\beta_{j_s}) = \beta_{g(j_s)}$ holds.

Suppose that $\{\beta'_{i_1}, \beta'_{i_2}, \dots, \beta'_{i_l}\} \subset\{\beta_{j_1}, \beta_{j_2}, \dots, \beta_{j_{n-l}}\}$, but we need a set of $n$ different points. Note that for any $a\in K$, the set $\{a\beta'_{i_1}, a\beta'_{i_2}, \dots, a\beta'_{i_l}\}$ also satisfies the necessary condition, namely, for any $g\in G$ and for any $s \in \{1,2,\dots, l\}$, the equality $g(a\beta'_{j_s}) = a\beta'_{g(j_s)}$ holds. Thus, it is enough for us to find such a nonzero scalar $a\in K$ that $a\beta'_{i_1}\notin\{\beta_{j_1}, \beta_{j_2}, \dots, \beta_{j_{n-l}}\}$, then $$\{\beta_{j_1}, \beta_{j_2}, \dots, \beta_{j_{n-l}}\}\cap\{a\beta'_{i_1}, a\beta'_{i_2}, \dots, a\beta'_{i_l}\} = \varnothing.$$
In this case we put  $B = \{\beta_{j_1}, \beta_{j_2}, \dots, \beta_{j_{n-l}}\}\cup \{a\beta'_{i_1}, a\beta'_{i_2}, \dots, a\beta'_{i_l}\}$ and we are done.

Suppose that there is no such nonzero scalar $a\in K$ that $a\beta'_{i_1}\notin\{\beta_{j_1}, \beta_{j_2}, \dots, \beta_{j_{n-l}}\}$, then, the opposite is true: for any nonzero scalar $a\in K$, the element $a\beta'_{i_1}$ is contained in $\{\beta_{j_1}, \beta_{j_2}, \dots, \beta_{j_{n-l}}\}$. By the condition of the lemma, we know that there are at least $k$ nonzero scalars in the field $K$. This means that the set $\{\beta_{j_1}, \beta_{j_2}, \dots, \beta_{j_{n-l}}\}$ contains at least $k$ elements proportional to $\beta'_{i_1}$. Each of these elements lies in its own unique orbit of length $l$, which means that among the orbits of the action of $G$ on the set $\{j_1, j_2, \dots, j_{n - l}\}$ there are at least $k$ orbits of length $l$. Therefore, among the orbits of the action of the group $G$ on the entire set $\{1,2,\dots,n\}$ there are at least $k+1$ orbits of length $l$, which contradicts the condition $c(G) = k$.

So, we proved that there exists such a nonzero scalar $a\in K$ that $$\{\beta_{j_1}, \beta_{j_2}, \dots, \beta_{j_{n-l}}\} \cap\{a\beta'_{i_1}, a\beta'_{i_2}, \dots, a\beta'_{i_l}\} = \varnothing.$$ Denote $$\beta_{i_1} = a\beta'_{i_1},\quad\beta_{i_2} = a\beta'_{i_2},\quad\dots,\quad\beta_{i_l} = a\beta'_{i_l},$$ then we put $$ B = \{\beta_{j_1}, \beta_{j_2}, \dots, \beta_{j_{n-l}}\}\cup\{\beta_{i_1}, \beta_{i_2}, \dots, \beta_{i_l}\} = \{\beta_1, \beta_2, \dots, \beta_n\}$$ and for any $g\in G$ and for any $i\in \{1,2,\dots, n\}$ the equality $g(\beta_i) = \beta_{g(i)}$ holds.\end{proof}

\begin{sled}\label{sled:aff} 
Let $n$ and $k$ be natural numbers, let $G$ be a group embedded in~$\mathfrak{S}_n$ of complexity $k$. Suppose a field $K$ has at least $k+1$ elements. If there exists a Galois extension $L\supset K$ with Galois group isomorphic to $G$, then there is a set of $n$ points $\{\beta_1, \beta_2, \dots, \beta_n\}$ on the affine line over $L$ such that the action of the Galois group $\Gal(K^{\sep}/K)$ on this set is equivalent to the action of the group $G$ on the set $\{1,2,\dots,n\}$.
\end{sled}

\begin{proof}
Let's use Lemma~\ref{lemma:g(b_i)}, we get a set of points $\{\beta_1, \beta_2, \dots, \beta_n\}$ on $\mathbb{A}^1_L$, in particular on $\mathbb{A}^1_{K^{\sep}}$. From the condition that for any $g\in G$ and for any $i\in \{1,2,\dots, n\}$ we have an equality $g(\beta_i) = \beta_{g(i)}$, it immediately follows that the actions of the group $G$ on the set $\{1,2,\dots,n\}$ and on the set $\{\beta_1, \beta_2, \dots, \beta_n\}$ are equivalent.

For any $i \in \{1,2,\dots, n\}$ element $\beta_i$ lies in the field $L$, thus, the action of the group $\Gal(K^{\sep}/K)$ on the set $\{\beta_1, \beta_2, \dots, \beta_n\}$ is induced from the action of the group $\Gal(L/K)\simeq G$ via the homomorphism $\Gal(K^{\sep}/K)\xrightarrow{}\Gal(L/K)$, hence it is equivalent to the action of the group $G$ on the set $\{1,2,\dots,n\}$. \end{proof}

\begin{sled}\label{sled:P1} 
The statement of Corollary~\ref{sled:aff} remains true if we replace the affine line with the projective one.
\end{sled}
\begin{proof}
We use Corollary~\ref{sled:aff}, and then embed the affine line into the projective line.\end{proof}

\begin{sled}\label{sled:P2} 
Let $n$ and $k$ be natural numbers, let $G$ be a group embedded in~$\mathfrak{S}_n$ of complexity $k$. Suppose a field $K$ has at least $k+1$ elements. If there exists a Galois extension $L\supset K$ with Galois group isomorphic to $G$, then there is a smooth conic $C$ in the projective plane $\mathbb{P}^2_{K^{\sep}}$ and a set of $n$ points on $C$, which is invariant with respect to the action of the group $\Gal(K^{\sep}/K)$, such that the action of the Galois group $\Gal(K^{\sep}/K)$ on this set is equivalent to the action of the group $G$ on the set $\{1,2,\dots,n\}$.
\end{sled}

\begin{proof}
Use Corollary~\ref{sled:P1} and isomorphically embed the projective line into the projective plane as a smooth conic.\end{proof}

\section{Realization of types of del~Pezzo surfaces of degree 5 over infinite fields}\label{inf fields}

Now we are ready to prove Theorem~\ref{theo:theo} on the realization of del~Pezzo surfaces of degree $5$ of various types over infinite fields.

\begin{proof}
The proof of necessity is given in Lemma~\ref{lemma:nec2}. Note that in Lemma~\ref{lemma:nec2} there are no restrictions on the original field, so in fact the necessity is proved for finite fields too.

Now we prove the sufficiency of the condition.

We are given a subgroup $H\subset\mathfrak{S}_5$ and we know that the field $K$ has a Galois extension $L\supset K$ with Galois group isomorphic to $H$. We want to construct a del Pezzo surface of degree $5$ of type $[H]$.

To do this, we use Corollary~\ref{sled:P2} for the case of $n= 5$, $k=c(H)$, $G=H$ and the construction described in Section~\ref{mainc}. The resulting surface is a del Pezzo surface of degree $5$ of type $[H]$.\end{proof}

\begin{proof}[of Corollary~\ref{sled:Q()}]
It is enough to prove that for any subgroup $H\subset\mathfrak{S}_5$ there exists a Galois extension of fields $\mathbb{F}'\supset\mathbb{F}$ with Galois group isomorphic to $H$, and use Theorem~\ref{theo:theo} for the proven case of an infinite field.

Note that all subgroups of the group $\mathfrak{S}_5$ are solvable except $\mathfrak{A}_5$ and $\mathfrak{S}_5$. The existence of a Galois extension of a number field with a given solvable Galois group is proved in the work of I. R. Shafarevich~\cite[Theorem 7]{Shaf}.

The proof of existence of Galois extensions of a number field with Galois groups isomorphic to $\mathfrak{A}_5$ and $\mathfrak{S}_5$ is written in the paper~\cite[§1]{NV}.\end{proof}

\section{Realization of types of del~Pezzo surfaces of degree 5 over finite fields}\label{finite field}

In fact, the case of a finite field is not much different from the case of an infinite field. For example, for fields with seven or more elements, the proof of Theorem~\ref{theo:theo} literally coincides with the proof from the previous section. In the case of fields of two, three, four and five elements, there are some nuances, but the theorem remains true.

\begin{proof}

The proof of the necessary condition of the theorem for finite fields, as already discussed, is exhausted by Lemma~\ref{lemma:nec2}, so we only need to prove the sufficiency.

So, suppose we have a subgroup $H\subset\mathfrak{S}_5$ and also we know that the field $K$ has a Galois extension $L\supset K$ with Galois group isomorphic to $H$. We want to construct a del Pezzo surface of degree $5$ of type $[H]$. We need to consider the following five cases:

\begin{enumerate}
    \item $|K|> c(H)$. Use Corollary~\ref{sled:P2} for the case $n= 5$, $k=c(H)$, $G=H$ and the construction described in Section~\ref{mainc}. The resulting surface will be a del Pezzo surface of degree 5 of type $[H]$.
    \item $|K|=5=c(H)$. It follows from the equality $c(H) = 5$ that the group $H$ is trivial. The construction of a del Pezzo surface of degree 5 of type $[e]$ is described in the Example~\ref{ex:1}.
    \item $|K|=4\leqslant c(H)$. It follows from the inequality $c(H)\geqslant 4$ that the group $H$ is trivial. The construction of a degree 5 del Pezzo surface of type $[e]$ is described in the Example~\ref{ex:1}.
    \item $|K|=3\leqslant c(H)$. It follows from the inequality $c(H)\geqslant 3$ that $H$ is either trivial subgroup or a subgroup of order 2 generated by transposition. Constructions of del Pezzo surfaces degree 5 of type $[e]$ and type $[\langle(1,2)\rangle]$ are described in the Examples~\ref{ex:1} and~\ref{ex:2}.
     \item $|K|=2\leqslant c(H)$. Since $H$ is the Galois group of a finite extension of a finite field, then $H$ is cyclic. The inequality $c(H)\geqslant 2$ holds for the following cyclic subgroups:
     \begin{itemize}
         \item $H$ is trivial subgroup. The construction of a del Pezzo surface of degree $5$ of type $[e]$ is described in the Example~\ref{ex:1}.
         \item $H$ is a subgroup of order 2 generated by transposition. Construction of a del Pezzo surface of degree $5$ of type $[\langle(1,2)\rangle]$ is described in the Example~\ref{ex:2}.
         \item $H$ is a subgroup of order 2 generated by a product of two disjoint transpositions. The construction of a del Pezzo surface of degree $5$ of type $[\langle(1,2)(3,4)\rangle]$ is described in the Example~\ref{ex:3}.
         \item $H$ is a subgroup generated by a triple cycle. Let's construct a surface of type $[\mathbb{Z}/3\mathbb{Z}]$. Using Corollary~\ref{sled:P2} for the case of $n=3$, $k=2$ and $G= \mathbb{Z}/3\mathbb{Z}$, we obtain three points in general position on the projective plane such that the Galois group $\text{Gal}(K^{\sep}/K)$ acts on this triple as $\mathbb{Z}/3\mathbb{Z}$. According to the Example~\ref{ex:5}, we conclude that there exists a del Pezzo surface of degree $5$ of type~$[\mathbb{Z}/3\mathbb{Z}]$ over $K$.
         \end{itemize}
\end{enumerate}
Now all cases are considered, and thus Theorem~\ref{theo:theo} is proved in full generality without any restrictions on the base field.\end{proof}

\section{Automorphism groups of del Pezzo surfaces of degree 5}\label{aut groups}

We have dealt with the criteria for the existence of del Pezzo surfaces of degree 5 of various types over a given field, and now it is natural to ask about the automorphism groups of these surfaces.

Let $X$ be a del Pezzo surface of degree 5 of type $[H]$ over a field $K$. We have noticed many times that the group $\text{Gal} (K^{\sep}/K)$ acts on the graph of $(-1)$-curves by automorphisms, the corresponding homomorphism to the group $\mathfrak{S}_5$ we denoted by $h$. By Lemma~\ref{lemma:freeaction} the group $\text{Aut}(X)$ also acts on the graph of $(-1)$-curves by automorphisms. Denote the corresponding homomorphism to the group $\mathfrak{S}_5$ by~$\psi$, from Corollary~\ref{sled:S_5} we remember that $\psi$ is injection.

\begin{lemma}\label{lemma:centralaizer}
Subgroup $\psi(\emph{Aut}(X))$ lies in the centralizer of the subgroup $h(\emph{Gal} ( K^{\sep}/K))$.
\end{lemma}

\begin{proof}
Automorphisms of the surface $X$ commute with the action of the Galois group $\text{Gal}(K^{\sep}/K)$ on this surface. Therefore, the subgroups $\psi(\text{Aut}(X))$ and $h(\text{Gal} ( K^{\sep}/K))$ commute, i.e. $\psi(\text{Aut}(X))$ lies in the centralizer of the subgroup $h(\text{Gal} ( K^{\sep}/K))$.\end{proof}

\begin{notation}\label{not: X over K}
Let $H$ be a subgroup in the group $G$. We denote by $C_G(H)$ the centralizer of the group $H$ in the group $G$.
\end{notation}

Now we are ready to prove Theorem~\ref{theo:autgrp}.

\begin{proof}[of Theorem~\ref{theo:autgrp}] 
Let $X$ be a del Pezzo surface of degree 5 of type $[H]$ over a field $K$. Without loss of generality, we can assume that $h(\text{Gal} ( K^{\sep}/K)) = H$. From Lemma~\ref{lemma:centralaizer} we know that $\psi(\text{Aut}(X))\subset C_{\mathfrak{S}_5}(H)$. It remains to show that these groups actually coincide.
1
Consider the permutation $\sigma\in C_{\mathfrak{S}_5}(H)$. Since each $(-1)$-curve on $X_{\overline{K}}$ is defined over $K^{\sep}$, then there exists an automorphism $f$ of the surface $X_{K^{\sep}}$ such that $\psi(f) = \sigma$. Denote by $l_1$, $l_2$, $\dots$, $l_{10}$ all the $(-1)$-curves on the surface $X_{K^{\sep}}$. Since $\sigma$ lies in the centralizer of the subgroup $H$, then for any $\gamma\in\Gal(K^{\sep}/K)$ and for any $i\in \{1,2,\dots, 10\}$ the equality $$f(\gamma(l_i)) = \gamma(f(l_i)) = \gamma(f)(\gamma(l_i))$$ holds.
Here $\gamma(f)$ is an automorphism obtained by applying the element $\gamma$ to the automorphism $f$. From the fact that the set $\{\gamma(l_i)\}_{i=1,2,\dots,10}$ coincides with the set $\{l_i\}_{i=1,2,\dots,10}$, we conclude that the automorphisms $f$ and $\gamma(f)$ induce the same automorphism of the graph of $(-1)$-curves. Then by Lemma~\ref{lemma:freeaction} these automorphisms coincide $$\gamma(f) = f.$$

Since the written equality holds for any $\gamma \in\Gal(K^{\sep}/K)$, then using the equality $$\text{Aut}(X) = \text{Aut}(X_{K^{\sep}})^{\Gal(K^{\sep}/K)},$$ we conclude that $f$ is defined over $K$. Therefore, each permutation $\sigma\in C_{\mathfrak{S}_5}(H)$ is realized by an automorphism of the surface $X$, that is, $\psi(\text{Aut}(X)) = C_{\mathfrak{S}_5}(H)$. Since $\psi$ is injective, we get the required result $\text{Aut}(X)\simeq C_{\mathfrak{S}_5}(H)$.\end{proof}

\begin{sled}\label{sled:list of auts}
Let $X$ be a del Pezzo surface of degree~$5$. Then $X$ has the following automorphism groups, depending on the type.

\begin{center}
\begin{tabular}{ | c | c | }
\hline
Type of $X$ & Automorphism group \\ \hline
$[e]$ & $\mathfrak{S}_5$ \\
$[\langle(1,2)\rangle]$ & $\mathfrak{S}_3 \times \mathbb{Z}/2\mathbb{Z}$ \\
$[\langle(1,2)(3,4)\rangle]$  & $\mathrm{D}_4$\\
$[\langle(1,2),(3,4)\rangle]$ & $\mathbb{Z}/2\mathbb{Z}\times \mathbb{Z}/2\mathbb{Z}$\\
$[\langle(1,2)(3,4),(1,3)(2,4)\rangle]$ & $\mathbb{Z}/2\mathbb{Z}\times \mathbb{Z}/2\mathbb{Z}$\\
$[\mathbb{Z}/3\mathbb{Z}]$ & $\mathbb{Z}/6\mathbb{Z}$\\
$[\mathbb{Z}/6\mathbb{Z}]$ & $\mathbb{Z}/6\mathbb{Z}$\\
$[\mathbb{Z}/4\mathbb{Z}]$ & $\mathbb{Z}/4\mathbb{Z}$\\
$[\mathbb{Z}/5\mathbb{Z}]$ & $\mathbb{Z}/5\mathbb{Z}$\\
$[\langle(1,2,3),(1,2)\rangle]$ & $\mathbb{Z}/2\mathbb{Z}$\\
$[\langle(1,2,3),(1,2)(4,5)\rangle]$ & $\mathbb{Z}/2\mathbb{Z}$\\
$[\mathrm{D_4}]$ & $\mathbb{Z}/2\mathbb{Z}$\\
$[\mathfrak{S}_3\times\mathbb{Z}/2\mathbb{Z}]$ & $\mathbb{Z}/2\mathbb{Z}$\\
$[\mathfrak{S}_5]$, $[\mathfrak{A}_5]$, $[\mathfrak{S}_4]$, $[\mathfrak{A}_4]$, $[\mathrm{D_5}]$, $[\mathrm{GA}(1,5)]$ & $e$\\
\hline
\end{tabular}
\end{center}
\end{sled}
\begin{proof}

To prove it, you need to use Theorem~\ref{theo:autgrp} and calculate     centralizers $C_{\mathfrak{S}_5}(H)$ for all types of surfaces.\end{proof}

Recall that the del Pezzo surface $X$ with the action of the group $G$ is called \emph{$G$-minimal} if $\text{rkPic}(X)^G=1$.

\begin{lemma}[cf.~{\cite[Theorem 6.4]{DI1}}]\label{lemma:lemma}
Let $X$ be a del Pezzo surface of degree 5 over a field $K$, and let $G$ be a subgroup in $\emph{Aut}(X)\subset\mathfrak{S}_5$. The surface $X$ is $G$-minimal if and only if the group $\Delta$ generated by the subgroup $G$ and the image of the group $\Gal(K^{\sep}/K)$ under the homomorphism $h$ contains a subgroup isomorphic to~\mbox{$\mathbb{Z}/5\mathbb{Z}$}.
\end{lemma}

\begin{proof}
Consider the exact sequence of groups \hbox{(see~\cite[Exercise 3.3.5(iii)]{G-S})}: $$0 \xrightarrow{} \text{Pic}(X) \xrightarrow{} \text{Pic}(X_{K^{\sep}})^{\Gal(K^{\sep}/K)} \xrightarrow{} \text{Br}(X) \xrightarrow{} \text{Br}(K(X)).$$ Now, since there is a $K$-point on $X$ (see, for example,~\cite{S2}), then the last homomorphism is an embedding, hence $$\text{Pic}(X)\simeq\text{Pic}(X_{K^{\sep}})^{\Gal(K^{\sep}/K)} = \text{Pic}(X_{K^{\sep}})^{h(\Gal(K^{\sep}/K))}.$$ Passing to $G$-invariants, we obtain an isomorphism: $$\text{Pic}(X)^G \simeq \text{Pic}(X_{K^{\sep}})^{\Delta}.$$

If $X$ is $G$-minimal, then by Lemma~\ref{lemma:no G-inv vert} the subgroup $\Delta$ contains a group isomorphic to $\mathbb{Z}/5\mathbb{Z}$.

Conversely, if the subgroup $\Delta$ contains a group $\mathbb{Z}/5\mathbb{Z}$, then $\text{rk}(\text{Pic}(X_{K^{\sep}})^{\Delta}) = 1$, because a five-dimensional representation of the group $\mathbb{Z}/5\mathbb{Z}$ over $\mathbb{Q}$, which has a one-dimensional trivial subrepresentation, is either trivial itself (this is not our case), or contains exactly one one-dimensional trivial subrepresentation. Therefore, $X$ is $G${\Russian"~}minimal.\end{proof}

From Lemma~\ref{lemma:centralaizer} and Lemma~\ref{lemma:lemma} we get a corollary.

\begin{sled}
The automorphism group of the minimal del Pezzo surface of degree 5 is either trivial or isomorphic to $\mathbb{Z}/5\mathbb{Z}$.
\end{sled}

Now let's prove Proposition~\ref{predl:G-min}.

\begin{proof}[of Proposition~\ref{predl:G-min}]
Let's prove the necessity. Suppose that there exists a $G${\Russian"~}minimal del Pezzo surface of degree $5$ over a field $K$. From Lemma~\ref{lemma:lemma} we know that the group $\Delta$ generated by the group $G$ and the image of the group $\Gal(K^{\sep}/K)$ under the homomorphism $h$, contains a subgroup isomorphic to~\mbox{$\mathbb{Z}/5\mathbb{Z}$}. Since the subgroups $G\subset\mathfrak{S}_5$ and $h(\Gal(K^{\sep}))\subset\mathfrak{S}_5$ commute, then there is a surjective homomorphism $$p\colon G\times h(\Gal(K^{\sep})) \twoheadrightarrow \Delta.$$

Therefore, either $G$ contains a subgroup isomorphic to $\mathbb{Z}/5\mathbb{Z}$, or $h(\Gal(K^{\sep}))$ contains a subgroup isomorphic to $\mathbb{Z}/5\mathbb{Z}$. In the first case, condition 1 of Proposition~\ref{predl:G-min} holds. In the second case, by Theorem~\ref{theo:theo} condition 2 holds. Moreover, by Theorem~\ref{theo:autgrp}, subgroup $G$ is contained in the centralizer $C_{\mathfrak{S}_5}(h(\Gal(K^{\sep})))$, which contains in the centralizer $C_{\mathfrak{S}_5}(\mathbb{Z}/5\mathbb{Z}) \simeq \mathbb{Z}/5\mathbb{Z}$. If condition 1 does not hold, that is, $G$ does not contain $\mathbb{Z}/5\mathbb{Z}$, then $G$ is trivial.

Now let's prove the sufficiency. Suppose condition 1 holds. Then the del Pezzo surface of degree 5 of type $[e]$ will be $G$-minimal, and by Theorem~\ref{theo:theo} surface of type $[e]$ exists over any field ($G$-minimality follows from Lemma~\ref{lemma:lemma}). Assume that condition 2 holds. Then by Theorem~\ref{theo:theo} there exists a del Pezzo surface of degree $5$ of type $[\text{Gal}(L/K)]$ over~$K$. By Lemma~\ref{lemma:lemma}, the surface $X$ will be $G$-minimal with respect to the action of the trivial group $G$.\end{proof}

\section{Classification of del Pezzo surfaces of degree~6}\label{dP6}

Consider an arbitrary field $K$. Let $X$ be a del Pezzo surface of degree 6 over $K$. The Galois group $\text{Gal} ( K^{\sep}/K )$ acts on the graph of $(-1)$-curves by automorphisms. Since the automorphism group of the graph of $(-1)$-curves of $X_{K^{\sep}}$ is isomorphic to the group $\mathfrak{S}_3\times\mathbb{Z}/2\mathbb{Z}$, see~\mbox{\cite[Theorem~8.4.2]{D}}, we get the homomorphism $$h\colon\text{Gal}(K^{\sep}/K) \xrightarrow[]{}\mathfrak{S}_3\times\mathbb{Z}/2\mathbb{Z}.$$

Let $H$ be a subgroup in $\mathfrak{S}_3\times\mathbb{Z}/2\mathbb{Z}$, denote by $[H]$ its conjugacy class.

\begin{definition}\label{def:type6} We say that a del Pezzo surface of degree $6$ has type $[H]$ if  $\Im{h} \in [H]$.
\end{definition}

It is easy to see that there are only ten types of del Pezzo surfaces of degree 6, namely: $[e]$, $[\langle((1,2),0)\rangle]$, $[\langle((1,2),1)\rangle]$, $[\langle(\text{id},1)\rangle]$, \hbox{$[\mathbb{Z}/2\mathbb{Z} \times\mathbb{Z}/2\mathbb{Z}]$}, $[\mathbb{Z}/3\mathbb{Z}]$, $[\mathbb{Z}/6\mathbb{Z}]$, $[\langle((123),0), ((12),0)\rangle]$, $[\langle((123),0), ((12),1)\rangle]$, \hbox{$[\mathfrak{S}_3\times\mathbb{Z}/2\mathbb{Z}]$}, cf.~\cite[Figure 1]{S-Z}.

\begin{proof}[of Theorem~\ref{theo:theo2}]
The proof of the necessity of condition is absolutely analogous to the proof of the necessity of condition of Theorem~\ref{theo:theo}. Therefore, we immediately move on to sufficiency.

Suppose we have a subgroup $H \subset \mathfrak{S}_3\times\mathbb{Z}/2\mathbb{Z}$ and also we know that the field $K$ has a Galois extension $L\supset K$ with Galois group isomorphic to $H$. We want to construct a del Pezzo surface of degree $6$ of type $[H]$. To do this, we first construct a del Pezzo surface of degree 5 of the corresponding type. 

Consider an embedding $$i\colon\mathfrak{S}_3\times\mathbb{Z}/2\mathbb{Z}\hookrightarrow \mathfrak{S}_5.$$
Denote by $G$ the image of this homomorphism, put $H' = i(H)$ and construct a del Pezzo surface of degree 5 of type $[H']$. The existence of such a surface is guaranteed by Theorem~\ref{theo:theo}. Since $H'$ is a subgroup in $G$, then by Lemma~\ref{lemma:dP5->dP6} there is a vertex $v$ of the graph $\Gamma$ that is invariant with respect to the action of the group $H'$. So $(-1)${\Russian"~}curve corresponding to the vertex $v$ is invariant with respect to the action of the group $\text{Gal} (K^{\sep}/K)$, therefore, defined over $K$. Blowing down this $(-1)$-curve, we get a del Pezzo surface of degree $6$ of type~$[H]$.\end{proof}

The classification of rational del Pezzo surfaces of degree $6$ (in particular, $G$-minimal ones) over perfect fields and information about their automorphism groups can be found in \cite[Section 4]{S-Z}.


\begin{thebibliography}{99}

\bibitem{AB} A. Auel and M. Bernardara, Semiorthogonal decompositions and birational geometry of del Pezzo surfaces over arbitrary fields. Proc. Lond. Math. Soc. (3) 117 (2018).

\bibitem{Beauv} A. Beauville, Finite subgroups of ${\rm PGL}_2(K)$. Vector bundles and complex geometry, 23–29, Contemp. Math., 522, Amer. Math. Soc., Providence, RI, 2010.

\bibitem{D} I. V. Dolgachev, Classical algebraic geometry: a modern view.~--- Cambridge: Cambridge university press, 2012.

\bibitem{DI1} I. V. Dolgachev and V. A. Iskovskikh, Finite subgroups of the plane Cremona group, volume 269 of Progr. Math., pages 443–548. Birkh\"{a}user Boston, Inc., Boston, MA, 2009.

\bibitem{DI2} I. V. Dolgachev and V. A. Iskovskikh, On elements of prime order in the plane Cremona group over a perfect field, International Mathematics Research Notices, Volume 2009, Issue 18, 2009.

\bibitem{G-A} M. Garcia-Armas, Finite group actions on curves of genus zero. J. Algebra 394 (2013).

\bibitem{G-S} S. O. Gorchinskiy and C. A. Shramov, Unramified Brauer group and its applications. Translations of Mathematical Monographs, 2015.

%\bibitem{Kummer} S. Lang, Algebra, Revised 3rd ed., Springer-Verlag, New York, Berlin, Heidelberg, 2002.

\bibitem{SLZ} H.-Y. Lin, E. Shinder, S. Zimmermann, Factorization centers in dimension two and the Grothendieck ring of varieties. arXiv:2012.04806, 2020.

\bibitem{Serre} J.-P. Serre, A Minkowski-style bound for the orders of the finite subgroups of the Cremona group of rank 2 over an arbitrary field, Mosc. Math. J., 9:1 (2009).

%\bibitem{SV} C. Shramov and V. Vologodsky, Automorphisms of pointless surfaces. arXiv:1807.06477, 2020.

\bibitem{S-Z} J. Schneider and S. Zimmermann - Algebraic subgroups of the plane Cremona group over a perfect field. Épijournal de Géométrie Algébrique, 16 novembre 2021, Volume~5.

\bibitem{Shaff} I. Shafarevich. Basic algebraic geometry. 1. Varieties in projective space. Third edition. Translated from the 2007 third Russian edition. Springer, Heidelberg, 2013.

\bibitem{Shaf} I.~R.~Shafarevich, Construction of fields of algebraic numbers with given solvable Galois group. Izv. Akad. Nauk SSSR Ser. Mat, 1954.

\bibitem{Skoro} A. N. Skorobogatov, Torsors and rational points. Cambridge Tracts in Mathematics. Cambridge: Cambridge University Press, 2001. 

\bibitem{S2} A. N. Skorobogatov, On a theorem of Enriques --- Swinnerton-Dyer. Annales de la Faculté des sciences de Toulouse: Mathématiques, Série 6, Tome 2 (1993).

\bibitem{prim} B. L. van der Waerden, Modern algebra. 1953.

\bibitem{NV} N. Vila, On the inverse problem of Galois theory. Publ. Mat. 36 (1992), no. 2B, 1053–1073 (1993).

\bibitem{Y} E. A. Yasinsky, Automorphisms of real del Pezzo surfaces and the real plane Cremona group. arXiv: 1912.10980, 2019.

\bibitem{stacks-project} The Stacks project. \href{https://stacks.math.columbia.edu}{\texttt{https://stacks.math.columbia.edu}}

\bibitem{act} \href{https://mathoverflow.net/questions/167916/incomplete-failures-of-the-inverse-galois-problem/167917}{\texttt{https://mathoverflow.net/questions/167916/incomplete-failures-of-the-inverse-
galois-problem/167917}}

\end{thebibliography}
\end{document}